\documentclass[a4,11pt, leqno]{article}
\usepackage{amsmath,amsfonts,amsthm,latexsym,amssymb,mathrsfs}
\usepackage{bbm}

\def\L{\mathcal{L}}
\def\X{\mathcal{X}}
\def\K{\mathcal{K}}
\def\C{\mathcal{C}}
\def\R{\mathcal{R}}
\def\H{\mathcal{H}}

\newtheorem{df}{Definition}[section]
\newtheorem{thm}[df]{Theorem}

\newtheorem{cor}[df]{Corollary}

\newtheorem{lem}[df] {Lemma}

\begin{document}
\title{\bf An addendum to: Analytically Riesz operators and Weyl and Browder type theorems}
\author{Enrico Boasso}

\date{  }

\maketitle

\begin{abstract}\noindent In this note a characterization of anallytically Riesz operators is given. 
This work completes the article \cite{B}.
\end{abstract}

\section{Introduction}

\noindent Anallytically Riesz operators, i.e., bounded and linear maps $T$ defined on a Banach space $\X$ such that there exists an analytical
function $f$ defined on a neighbourhood of the spectrum of $T$ with the property that $f(T)$ is Riesz, were studied in \cite{KS}. In addition, several spectra and some spectral properties of this class of operators were studied in  \cite{B}. In particular, when instead of an analytical funcion there exists a polynomial
$P\in\mathbb C [X]$ such that $P(T)$ is Riesz, $T$ is said to be a polynomially Riesz operator. The structure of polynomially Riesz operators
was studied in \cite{Geo} (see also \cite[Theorem 2.13]{ZDHD}). To learn more about polynomially Riesz operators see for example \cite{ZDHD}
and its reference list.\par

\indent After the publication of \cite{B}, a characterization of analytically Riesz operators was obtained. In fact, similar arguments to the ones in 
\cite{B} prove that necessary and sufficient for $T$ to be analytically Riesz is that there exist $\X_1$ and $\X_2$  two closed and complemented $T$-invatiant subspaces of $X$ such that if $T=T_1\oplus T_2$, then $T_1\in\L (\X_1)$ is an arbitray operator
and $\X_2$ is finite dimensional or $T_2\in\L (\X_2)$ is a polynomially Riesz operator ($T_i=T\mid_{\X_i}$, $i=1, 2$). In other words,
an analytically Riesz operator is essentially the direct sum of an arbitray operator and a polynomially Riesz operator. The objective of this note is to present this characterization and some other related results. This note consists in a completation of \cite{B}.    

\section{Results}

\indent Fron now on  $\X$ will denote an infinite dimensional complex Banach space, $\L(\X)$  the algebra of all bounded and linear maps
defined on and with values in $\X$ and $I\in\L (\X)$ the identity map. If $T\in \L(\X)$, then $N(T)$, $R(T)$ and $\sigma (T)$ will stand for the null space, the
range and the spectrum of $T$, respectively. In addition, $\K(\X)\subset \L (\X)$ will denote
the closed ideal of compact operators defined on $\X$, $\C (\X)$ the Calkin algebra of $\X$ and $\pi\colon \L (\X)\to\C (\X)$ the quotient map.\par

\indent Recall that $T\in L(\X)$  is said to be  a
\it Fredholm \rm operator if $\alpha(T)=\dim N(T)$ and $\beta(T)=\dim \X/R(T)$
are finite dimensional. In addition,  $T\in\L (\X)$ is said to be a \it Riesz operator\rm, if $T-\lambda I$ is Fredholm 
for all $\lambda\in\mathbb C$, $\lambda\neq 0$.
The set of all Riesz operators defined on $\X$ will be denoted by $\R(\X)$. 
More generally, $T$ will be said to be an \it analytically Riesz operator\rm, if there exists a holomorphic function $f$ defined on an open 
neighbourhood of $\sigma (T)$  such that $f(T)\in \R (\X)$
($\H (\sigma(T))$ will denote the algebra of germs of analytic functions defined on open neighbourhoods of $\sigma (T)$).
In particular, $T$ will be said to be \it polynomially Riesz\rm,
 if there exists $P\in C[X]$ such that $P(T)\in \R(\X)$ (see \cite{ZDHD}).

\indent To prove the main result of this note, a preliminary result is needed.\par

\begin{lem}\label{lemma1}Let $\X$ be an infinite dimensional complex Banach space and consider $T\in\R (\X)$. Then,
if $V\in\L (\X)$ is such that $TV-VT\in\K (\X)$, $VT$ and $TV\in \R (\X)$.
\end{lem}

\begin{proof}Recall that necessary and sufficient for $T\in\R (\X)$ is that $\pi (T)\in \C(\X)$ is quasinilpotent. Since
$\pi (V)$ commutes with $\pi (T)$, it is not difficult to prove that $\pi (V)\pi (T)=\pi (T)\pi (V)$ is quasinilpotent.
In particular, $VT$ and $TV\in \R (\X)$.
\end{proof}

\indent Next a characterization of analytically Riesz operators is given. Note that the following notation will be used.
If $\X_1$ and $\X_2$ are two closed and complemented $T$-invariant subspaces of the Banach space $\X$ ($T\in \L (\X)$),
then $T$ has the decomposition $T=T_1\oplus T_2$, where $T_i=T\mid_{\X_i}$, $i=1, 2$.

\begin{thm}\label{thm1}Let $\X$ be an infinite dimensional complex Banach space and consider $T\in\L (\X)$. Then, the following statements are
equivalent:\par
\noindent \rm (i) \it The operator $T$ is analytically Riesz.\par
\noindent \rm (ii) \it There exist $\X_1$ and $\X_2$ two closed and complementd $T$-invariant subspaces of $\X$ with the property that, 
if $T=T_1\oplus T_2$, then $T_1\in \L (\X_1)$ is an arbitrary operator and either $\X_2$ is finite dimensional or $T_2\in\L (\X_2)$ is polynomially Riesz.\par
\end{thm}
\begin{proof} Suppose that statement (i) holds.  Let $f\in\H (\sigma (T))$ be such that $f(T)\in\R (\X)$.  According to \cite[Theorem 1]{KS},
there are two closed disjoint sets $S_1$ and $S_2$ such that $\sigma (T)=S_1\cup S_2$, $f$ is  locally zero 
at each point of $S_1$ but it is not at any point of $S_2$. In addition,  if $\X_1$ and $\X_2$ are two closed and complemented $T$-invariant subspaces of $\X$ associated to the decompositon defined by $S_1$ and $S_2$, then
$f(T_1)=0$ and  either $\X_2$ is a finite dimensional space  
or $T_2$ can be decomposed as a direct sum of operators (see the proof of \cite[Theorem 1]{KS} for details).\par

\indent Suppose then that dim $\X _2$ is infinite. Let $f_2\in\H (\sigma (T_2))$ be defined as the restriction of $f$ to an open set  containing $S_2$ but disjoint to $S_1$ (recall that $\sigma (T_1)=S_1$ and  $\sigma (T_2)=S_2$). Since $f(T)=0\oplus f_2(T_2)$, $T_2\in \L (\X_2)$ is analytically Riesz.  
Since $f$ is analytically zero at no point of $S_2$, the set $f_2^{-1} (0)\cap S_2$ is finite. In particular, there exist $n\in\mathbb N$, $k_i\in \mathbb N$ and $\lambda_i\in f_2^{-1} (0)\cap S_2 $ ($i=1,\dots ,n$) such that $f_2(z)=(z-\lambda_1)^{k_1}\dots (z-\lambda_n)^{k_n}g(z)$, where 
$g\in \H (\sigma (T_2))$ is such that
$g(z)\neq 0$ for all $z\in\sigma (T_2)$. \par

\indent Now, since $(T_2-\lambda_1)^{k_1}\dots (T_2-\lambda_n)^{k_n}g(T_2)=f_2(T_2)\in\R (\X_2)$, and $g(T_2)\in\L (\X_2)$
is an invertible operator, which commutes with  $(T_2-\lambda_1)^{k_1}\dots (T_2-\lambda_n)^{k_n}$, according to Lemma \ref{lemma1}, $(T_2-\lambda_1)^{k_1}\dots (T_2-\lambda_n)^{k_n}\in\R (\X_2)$. 
Consequently, $T_2$ is polinomially Riesz.\par

\indent To prove the converse, consider two closed and disjoint sets $S_1$, $S_2\subset \mathbb C$
such that $S_1\cup S_2=\sigma (T)$, $\sigma (T_i)=S_i$, $i=1, 2$. Let $U_i$ be two disjoint open sets
such that $S_i\subset U_i$, $i=1, 2$, and define $f\in\H (\sigma (T))$ as follows: $f\mid_{U_1}=0$
and when $\X_2$ is finite dimensional, $f\mid_{U_2}=z$, while when dim $\X_2$ is infinite, $f\mid_{U_2}=P$, where $P\in\mathbb C[X]$ is such that $P(T)\in\R (\X)$. Therefore, $f(T)=0\oplus T_2$ (dim $\X_2<\infty$) or $f(T)=0\oplus P(T)$ (dim $\X_2$ infinite).
In both cases $T$ is analytically Riesz.\par

\end{proof}

\indent Recall that polynomially Riesz operators were characterized in \cite{Geo, ZDHD} (see in paricular
\cite[Theorems 2.2, 2.3, 2.13]{ZDHD}).
  
\begin{cor}\label{cor111}Let $\X$ be an infinite dimensional complex Banach space and consider $T\in\L (\X)$.
Suppose that there is $f\in\H (\sigma (T))$ such that $f(T)\in \R(\X)$ and for each $x\in\sigma (T)$,
$f$ is not locally zero at $x$. Then, $T$ is polynomially Riesz.
\end{cor}
\begin{proof} Using the same notation as in  Theorem \ref{thm1}, according to the proof of this Theorem and \cite[Theorem 1]{KS},
$S_1=\emptyset$ and $\X=\X_2$. Since dim $\X_2$ is infinite, $T$ is polynomially Riesz.
\end{proof}

\indent Note that according to Corollary \ref{cor111}, the operators considered in \cite{B} are polynomially Riesz.
Compare with \cite{ZDHD}.
\bibliographystyle{amsplain}

\vskip.3truecm
\noindent Enrico Boasso\par
\noindent E-mail address: enrico\_odisseo@yahoo.it
\end{document}